\documentclass[11 pt,final]{article}

\usepackage{amsmath,amssymb,amsthm,amscd,amsfonts}
\usepackage{epsfig,pinlabel}
\usepackage[hmargin=1.25in, vmargin=1in]{geometry}
\usepackage[font=small,format=plain,labelfont=bf,up,textfont=it,up]{caption}
\usepackage{subcaption}
\usepackage{stmaryrd}
\usepackage{verbatim}
\usepackage{type1cm}
\usepackage{url}
\usepackage{calc}
\usepackage[all,cmtip]{xy}
\usepackage{graphicx}
\usepackage{color}

\newcommand{\urlwofont}[1]{\urlstyle{same}\url{#1}}

\newcommand{\nc}{\newcommand}
\nc{\nt}{\newtheorem}
\nc{\dmo}{\DeclareMathOperator}

\theoremstyle{plain}
\newtheorem{theorem}{Theorem}[section]
\newtheorem{maintheorem}{Theorem}
\newtheorem{prop}[theorem]{Proposition}
\newtheorem{lemma}[theorem]{Lemma}

\theoremstyle{definition}

\theoremstyle{remark}

\dmo{\SMod}{SMod}
\dmo{\PMod}{PMod}
\dmo{\SHomeo}{SHomeo}
\dmo{\SI}{\mathcal{SI}}
\dmo{\SSp}{SSp}
\dmo{\PSp}{PSp}

\newcommand\Z{\ensuremath{\mathbb{Z}}}

\newcommand\N{\ensuremath{\mathbb{N}}}

\nc{\p}[1]{\noindent {\bf #1.}}
\nc{\margin}[1]{\marginpar{\scriptsize #1}}

\nc{\PartialIBases}{\mathfrak{IB}}
\nc{\PartialIBasesEx}{\widehat{\mathfrak{IB}}}
\nc{\PartialBases}{\mathfrak{B}}
\nc{\Building}{\mathfrak{T}}
\nc{\height}{\ensuremath{\text{ht}}}
\nc{\Poset}{\mathfrak{P}}
\nc{\Field}{\mathbb{F}}
\nc{\Link}{\ensuremath{\text{Link}}}
\nc{\Star}{\ensuremath{\text{Star}}}
\nc{\Aut}{\ensuremath{\text{Aut}}}

\nc{\SymTorelli}{\ensuremath{\mathcal{SI}}}
\nc{\BTorelli}{\ensuremath{\mathcal{BI}}}
\dmo{\Braid}{\ensuremath{B}}
\dmo{\PureBraid}{\ensuremath{PB}}
\nc{\Hyper}{\ensuremath{\iota}}

\nc{\BigFreeProd}{\mathop{\mbox{\Huge{$\ast$}}}}

\nc{\Quotient}{\ensuremath{\mathcal{Q}}}
\nc{\QuotientEx}{\ensuremath{\widehat{\mathcal{Q}}}}

\nc{\Presentation}[2]{\ensuremath{\text{$\langle #1$ $|$ $#2 \rangle$}}}
\nc{\SpGen}{\ensuremath{S_{\text{Sp}}}}
\nc{\SpRel}{\ensuremath{R_{\text{Sp}}}}
\nc{\QGen}{\ensuremath{S_{\mathcal{Q}}}}
\nc{\QRel}{\ensuremath{R_{\mathcal{Q}}}}
\nc{\PBs}{\ensuremath{T}}
\nc{\Qs}{\ensuremath{\overline{s}}}

\dmo{\PB}{PB}
\nc{\BIredg}{\mathcal{BI}_{2g+1}^{\text{red}}}
\nc{\BI}{\mathcal{BI}}
\dmo{\D}{D}
\dmo{\Stab}{Stab}
\dmo{\Surger}{Surger}
\nc{\I}{\mathcal{I}}

\nc{\spanmap}{span}

\nc{\genbygen}[2]{\premonoid{#1}{#2}}
\nc{\premonoid}[2]{#1 \circledcirc #2}
\nc{\monoid}[2]{#1 \odot #2}

\nc{\raag}{A_\Gamma}
\nc{\raagdelt}{A_\Delta}
\nc{\autraag}{\mathrm{Aut}(\raag)}
\nc{\outraag}{\mathrm{Out}(A_\Gamma)}
\nc{\autraagdelt}{\mathrm{Aut}(A_\Delta)}
\nc{\glk}{\mathrm{GL}(k,\mathbb{Z})}
\nc{\gln}{\mathrm{GL}(n,\mathbb{Z})}
\nc{\glkdelt}{\mathrm{GL}(k |\Delta| ,\mathbb{Z})}
\nc{\gldelt}{\mathrm{GL}(|\Delta| ,\mathbb{Z})}
\nc{\zkdelt}{\mathbb{Z}^{k|\Delta|}}
\nc{\aut}{\mathrm{Aut}}
\nc{\out}{\mathrm{Out}}
\nc{\join}{\mathcal{J}}
\nc{\pc}{\mathrm{PC}}
\nc{\lk}{\mathrm{lk}}
\nc{\st}{\mathrm{st}}
\nc{\inn}{\mathrm{Inn}}
\include{diagram}

\title{\vspace{-30pt} On the number of outer automorphisms of the automorphism group of a right-angled Artin group}

\author{Neil J. Fullarton\vspace{-6pt}}

\begin{document}

\newcounter{enumi_saved}

\maketitle

\begin{abstract}We show that for any natural number $N$ there exists a right-angled Artin group $\raag$ for which $\out(\autraag)$ has order at least $N$. This is in contrast with the cases where $\raag$ is free or free abelian: for all $n$, Dyer-Formanek and Bridson-Vogtmann showed that $\out(\aut(F_n))=1$, while Hua-Reiner showed $|\out(\aut(\Z^n))| \leq 4$. We also prove the analogous theorem for $\out(\out(\raag))$. These theorems fit into a wider context of algebraic rigidity results in geometric group theory.  We establish our results by giving explicit examples; one useful tool is a new class of graphs called \emph{austere graphs}.
\end{abstract}

\maketitle

\section{Overview}A finite simplicial graph $\Gamma$ with vertex set $V$ and edge set $E \subset V \times V$ defines the \emph{right-angled Artin group} $\raag$ via the presentation
\[ \langle v \in V \mid [v,w] = 1 \mbox{ if } (v,w) \in E \rangle . \] The class of right-angled Artin groups contains all finite rank free and free abelian groups, and allows us to interpolate between these two classically well-studied classes of groups.

A centreless group $G$ is \emph{complete} if the natural embedding $\inn(G) \hookrightarrow \aut(G)$ is an isomorphism. Dyer-Formanek \cite{DyerFormanek} showed that $\aut(F_n)$ is complete for $F_n$ a free group of rank $n \geq 2$, giving $\out(\aut(F_n)) = 1$. Bridson-Vogtmann \cite{BridsonVogtmann} later proved this for $n \geq 3$ using geometric methods, and showed that $\out(F_n)$ is also complete, as did Khramtsov \cite{Khramtsov}. Although $\aut(\Z^n) = \gln$ is not complete (its centre is $\Z / 2$), we observe similar behaviour for free abelian groups. Hua-Reiner \cite{HuaReiner} explicitly determined $\out(\gln)$; in particular, $|\out( \gln) | \leq 4$ for all $n$. In other words, for free or free abelian $\raag$, the orders of $\out(\aut(\raag))$ and $\out(\out(\raag))$ are both uniformly bounded above. The main result of this paper is that no such uniform upper bounds exist when $\raag$ ranges over all right-angled Artin groups.

\begin{maintheorem}\label{mt1}For any $N \in \N$, there exists a right-angled Artin group $\raag$ such that $|\out(\aut(\raag))| > N.$ Moreover, we may choose $\raag$ to have trivial or non-trivial centre. \end{maintheorem}
We also prove the analogous result regarding the order of $\out(\out(\raag))$.
\begin{maintheorem}\label{mt2}For any $N \in \N$, there exists a right-angled Artin group $\raag$ such that $\out(\out(\raag))$ contains a finite subgroup of order greater than $N$. \end{maintheorem}
Improving upon Theorem B, in joint work with Corey Bregman, we have exhibited right-angled Artin groups $\raag$ for which $\out(\outraag)$ is infinite; this work will appear in a forthcoming paper.

We remark that neither Theorem A nor B follows from the other, since in general, given a quotient $G / N$, the groups $\aut(G/N)$ and $\aut(G)$ may behave very differently.

Many of the groups that arise in geometric group theory display `algebraic rigidity', in the sense that their outer automorphism groups are small. The aforementioned results of Dyer-Formanek \cite{DyerFormanek}, Bridson-Vogtmann \cite{BridsonVogtmann} and Hua-Reiner \cite{HuaReiner} are examples of this phenomenon. Further examples are given by braid groups \cite{DyerGrossman} and many mapping class groups \cite{IvanovMcCarthy}, as these groups have $\Z / 2$ as their outer automorphism groups. Theorems A and B thus fit into a more general framework of the study of algebraic rigidity within geometric group theory.

We prove the three theorems by exhibiting classes of right-angled Artin groups over which the groups in question grow without bound. We introduce the notions of an \emph{austere graph} and a \emph{weakly austere graph} in Sections 2 and 4, respectively. These lead to tractable decompositions of $\aut(\raag)$ and $\out(\raag)$, which then yield numerous members of $\out(\autraag)$ and $\out(\out(\raag))$. Our methods do not obviously yield infinite order elements of $\out(\autraag)$; we discuss this further in Section 5.

\textbf{Outline of paper.} In Section 2, we recall the finite generating set of $\aut(\raag)$ and give the proof of Theorem B. Sections 3 and 4 contain two proofs of Theorem A; first for right-angled Artin groups with non-trivial centre, then for those with trivial centre. In Section 5, we discuss generalisations of this work, including the question of extremal behaviour of $\out(\autraag)$. The potential existence of infinite order members of $\out(\autraag)$ is also discussed in Section 5, along with some difficulties of approaching this question. The Appendix contains a calculation used in the proof of Proposition \ref{splitdecomp}.

\textbf{Acknowledgements.} The author thanks his PhD supervisor Tara Brendle for her guidance and insight, and also thanks Corey Bregman, Ruth Charney and Karen Vogtmann for helpful discussions.  The author is grateful for the hospitality of the Institute for Mathematical Research at Eidgen\"{o}ssische Technische Hochschule Z\"{u}rich, where part of this work was completed. The author also thanks Dan Margalit and an anonymous referee for helpful comments on an earlier draft of this paper. 

\section{Proof of Theorem B} Let $\Gamma$ be a finite simplicial graph with vertex set $V$ and edge set $E \subset V \times V$. We write $\Gamma = (V,E)$. We abuse notation and consider $v \in V$ as both a vertex and a generator of $\raag$. We will also often consider a subset $S \subseteq V$ as the full subgraph of $\Gamma$ which it spans. For a vertex $v \in V$, we define its \emph{link}, $\lk(v)$, to be the set of vertices in $V$ adjacent to $v$, and its \emph{star}, $\st(v)$, to be $\lk(v) \cup \{v\}$.

\textbf{The LS generators.} Laurence \cite{Laurence} and Servatius \cite{Servatius} gave a finite generating set for $\autraag$, which we now recall. We specify the action of the generator on the elements of $V$. If a vertex $v \in V$ is omitted, it is assumed to be fixed. There are four types of generators:
\begin{enumerate} \item \emph{Inversions}, $\iota_v$: for each $v \in V$, $\iota_v$ maps $v$ to $v^{-1}$.
\item \emph{Graph symmetries}, $\phi$: each $\phi \in \aut(\Gamma)$ induces an automorphism of $\raag$, which we also denote by $\phi$, mapping $v \in V$ to $\phi(v)$.
\item \emph{Dominated transvections}, $\tau_{xy}$: for $x,y \in V$, whenever $\lk(y) \subseteq \st(x)$, we write $y \leq x$, and say $y$ is \emph{dominated} by $x$ (see Figure \ref{domination}). In this case, $\tau_{xy}$ is well-defined, and maps $y$ to $yx$. The vertex $x$ may be adjacent to $y$, but it need not be.
\item \emph{Partial conjugations}, $\gamma_{c,D}$: fix $c \in V$, and select a connected component $D$ of $\Gamma \setminus \st(c)$ (see Figure \ref{pcex}). The partial conjugation $\gamma_{c,D}$ maps every $d \in D$ to $cdc^{-1}$.
\end{enumerate}

\begin{figure}
\centering
\begin{subfigure}{.5\textwidth}
  \centering
  \includegraphics[width=2.6in]{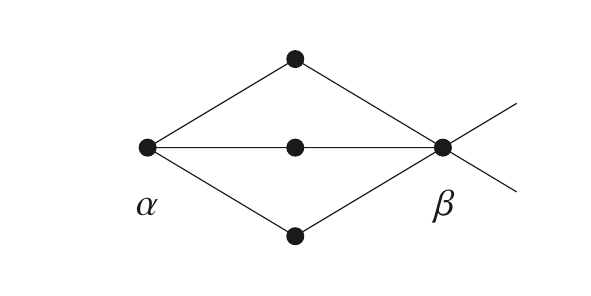}
  \caption{}
    \label{domination}
\end{subfigure}%
\begin{subfigure}{.5\textwidth}
  \centering
  \includegraphics[width=2.6in]{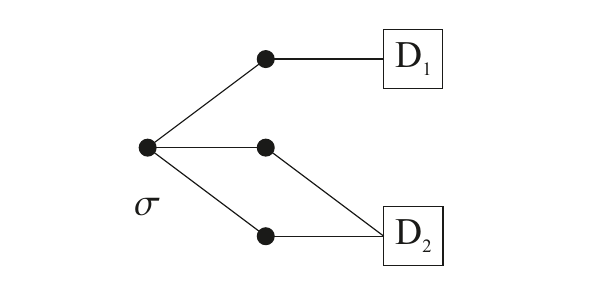}
  \caption{}
    \label{pcex}
\end{subfigure}
\caption{\textbf{\emph{(a)}} The local picture of a vertex $\alpha$ being dominated by a vertex $\beta$. \textbf{\emph{(b)}} Removing the star of the vertex $\sigma$ leaves two connected components, $\mathrm{D}_1$ and $\mathrm{D}_2$.}
\end{figure}
We refer to the generators on this list as the \emph{LS generators} of $\autraag$.

\textbf{Austere graphs.} We say that a graph $\Gamma = (V,E)$ is \emph{austere} if it has trivial symmetry group, no dominated vertices, and for each $v \in V$, the graph $\Gamma \setminus \st (v)$ is connected. We use examples of austere graphs to prove Theorem B.
 \begin{figure}[h!]
    \centering
    \includegraphics[width=2.6in]{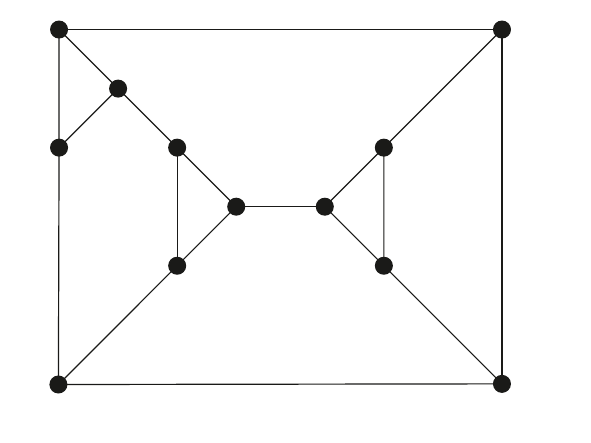}
    \caption{The Frucht graph, an example of a graph which is austere.}
    \label{frucht}
    \end{figure}
\begin{proof}[Proof of Theorem B] For an austere graph $\Gamma = (V,E)$, the only well-defined LS generators of $\autraag$ are the inversions and the partial conjugations. Let $n = |V|$. Note that each partial conjugation is an inner automorphism. We have the decomposition $$\autraag \cong \inn(\raag) \rtimes I_\Gamma,$$ where $I_\Gamma \cong (\Z /2)^n$ is the group generated by the inversions. The inversions act on $\inn(\raag) \cong \raag$ in the obvious way, either inverting or fixing (conjugation by) each $v \in V$. We have $\mathrm{Out}(\raag) \cong I_\Gamma$, and so $\mathrm{Aut(Out}(\raag)) \cong \mathrm{Out(Out}(\raag)) \cong \mathrm{GL}(n, \Z / 2).$ If we can find austere graphs for which $n$ is as large as we like, then we will have proved Theorem B.

The Frucht graph, seen in Figure \ref{frucht}, was constructed by Frucht \cite{Frucht} as an example of a 3-regular graph with trivial symmetry group. In fact, it is easily checked that the Frucht graph is austere. Baron-Imrich \cite{BaronImrich} generalised the Frucht graph to produce a family of finite, 3-regular graphs with trivial symmetry groups, over which $n = |V|$ is unbounded. Like the Frucht graph, these graphs may also be shown to be austere, and so they define a class of right-angled Artin groups which proves Theorem B. \end{proof}

\section{Proof of Theorem A: right-angled Artin groups with non-trivial centre}
In this section, we assume that $\raag$ has non-trivial centre. Let $\{\Gamma_i\}$ be a collection of graphs. The \emph{join}, $\join \{\Gamma_i\}$, of $\{\Gamma_i\}$ is the graph obtained from the disjoint union of $\{\Gamma_i\}$ by adding an edge $(v_i,v_j)$ for all vertices $v_i$ of $\Gamma_i$ and $v_j$ of $\Gamma_j$, for all $i \neq j$. Observe that for a finite collection of finite simplicial graphs $\{\Gamma_i \}$, we have $$A_{\join \{\Gamma_i\}} \cong \prod_i A_{\Gamma_i}.$$ When we take the join of only two graphs, $\Gamma$ and $\Delta$, we write $\join(\Gamma,\Delta)$ for their join.

\subsection{Decomposing $\autraag$} A vertex $s \in V$ is said to be \emph{social} if it is adjacent to every vertex of $V \setminus \{s\}$. Let $S$ denote the set of social vertices of $\Gamma$ and set $k = |S|$. Let $\Delta = \Gamma \setminus S$. We have $\Gamma = \join(S,\Delta)$, so $\raag \cong \Z^k \times \raagdelt$, and by \emph{The Centralizer Theorem} of Servatius \cite{Servatius}, the centre of $\raag$ is $A_S = \Z^k$.

No vertex $v \in \Delta$ can dominate any vertex of $S$ (otherwise $v$ would be social), and any $\phi \in \aut(\Gamma)$ must preserve $S$ and $\Delta$ as sets. Determining the LS generators, we see that $\autraag$ has $\glk \times \autraagdelt$ as a proper subgroup. The only LS generators not contained in this proper subgroup are of the form $\tau_{sa}$, where $s \in S$ and $a \in \Delta$. Note that this dominated transvection is defined for any pair $(s,a) \in S \times \Delta$. We will refer to this type of transvection as a \emph{lateral transvection}, as they occur `between' the two graphs, $S$ and $\Delta$.

\begin{prop}Let $\Gamma = \join(S,\Delta)$ define a right-angled Artin group, $\raag$, with non-trivial centre. The group $\mathcal{L}$ generated by the lateral transvections is isomorphic to $\zkdelt$. \end{prop}
\begin{proof}It is clear the lateral transvections $\tau_{sa}$ and $\tau_{tb}$ commute if $a \neq b$. The only case left to check is $\tau_{sa}$ and $\tau_{ta}$, for $s,t \in S$ and $a \in \Delta$. We see that
$$\tau_{ta}\tau_{sa}\tau_{ta}^{-1}(a) = \tau_{ta}\tau_{sa}(at^{-1}) = \tau_{ta}(ast^{-1}) =atst^{-1} = as,$$ since $s$ and $t$ commute. Therefore $\tau_{ta} \tau_{sa} \tau_{ta}^{-1} = \tau_{sa}$, and hence $\mathcal{L}$ is abelian. That it has no torsion follows from the fact that $\Z^k$ has no torsion. A straightforward calculation verifies that the lateral transvections form a $\Z$-basis for $\mathcal{L}$. To deduce the rank, observe there is a bijection between $\{\tau_{sa} \mid S \in S, a \in \Delta \}$ and $S \times \Delta$.  \end{proof}
 We now show that $\mathcal{L}$ is the kernel of a semi-direct product decomposition of $\autraag$. This is an $\autraag$ version of a decomposition of $\out(\raag)$ given by Charney-Vogtmann \cite{CharneyVogtmann}.
\begin{prop}\label{splitdecomp}Let $\Gamma = \join(S,\Delta)$ define a right-angled Artin group, $\raag$, with non-trivial centre. The group $\autraag$ splits as the product \[ \zkdelt \rtimes \left [ \glk \times \autraagdelt \right ] .\] \end{prop}

\begin{proof}
Standard computations show that $\mathcal{L} \cong \zkdelt$ is closed under conjugation by the LS generators: these calculations are summarised in the Appendix.  We observe that the intersection of $\mathcal{L}$ and $\glk \times \autraagdelt$ is trivial: the elements of $\mathcal{L}$ transvect vertices of $\Delta$ by vertices of $S$, whereas the elements of $\glk \times \autraagdelt$ carry $\Z^k$ and $\raagdelt$ back into themselves. Thus, $\autraag$ splits as in the statement of the proposition.
\end{proof}

We look to the $\zkdelt$ kernel as a source of automorphisms of $\autraag$. We must however ensure that the semi-direct product action is preserved; this is achieved using the theory of automorphisms of semi-direct products, which we now recall.

\textbf{Automorphisms of semi-direct products.} Let $G = N \rtimes H$ be a semi-direct product, where $N$ is abelian, with the action of $H$ on $N$ being encoded by a homomorphism $\alpha : H \to \mathrm{Aut}(N)$, writing $h \mapsto \alpha_h$. We will often write $(n,h) \in G$ simply as $nh$. Let $\mathrm{Aut}(G,N) \leq \mathrm{Aut}(G)$ be the subgroup of automorphisms which preserve $N$ as a set. For each $\gamma \in \mathrm{Aut}(G,N)$, we get an induced automorphism $\phi$, say, of $G / N$, and an automorphism $\theta$, say, of $N$, by restriction. The map $P : \mathrm{Aut}(G,N) \to \mathrm{Aut}(N) \times \mathrm{Aut}(H)$ given by $P(\gamma) = (\theta, \phi)$ is a homomorphism.

An element $(\theta, \phi) \in \aut(N)\times \aut(H)$ is said to be a \emph{compatible pair} if $\theta \alpha_h \theta^{-1} = \alpha_{\phi(h)},$ for all $h \in H$. Let $C \leq \aut(N)\times \aut(H)$ be the subgroup of all compatible pairs. This is a special (split, abelian kernel) case of the notion of compatibility for group extensions \cite{Passi}, \cite{Wells}. Notice that the image of $P$ is contained in $C$, since $\gamma \in \mathrm{Aut}(G,N)$ must preserve the relation $hnh^{-1} = \alpha_h(n)$ for all $h \in H, n \in N$. We therefore restrict the codomain of $P$ to $C$. Note that while $P$ (with its new codomain) is surjective, it need not be injective: consider, for example, any automorphism of $\Z \times \Z$ that preserves one copy of $\Z$ but not the other. We map $C$ back into $\aut(G,N)$ using the homomorphism $R$, defined by \[R(\theta,\phi)(nh) = \theta(n)\phi(h). \]

Let $\mathrm{Aut}_H(G,N)$ be the subgroup of $\mathrm{Aut}(G,N)$ of maps which induce the identity on $H$. This group is mapped via $P$ onto $$C_1 := \{\theta \in \mathrm{Aut}(N) \mid \theta \alpha(h) \theta^{-1} = \alpha(h) \phantom{\phi} \forall h \in H \}.$$
Note $C_1$ is the centraliser of $\mathrm{im}(\alpha)$ in $\mathrm{Aut}(N)$. We determine $C_1$ for the semi-direct product decomposition of $\autraag$ given by Proposition \ref{splitdecomp}, and use $R$ to map $C_1$ into $\aut(\autraag)$.

\subsection{Ordering the lateral transvections} In order to determine the image of $\alpha$ for our semi-direct product, $\zkdelt \rtimes \left [ \glk \times \autraagdelt \right ]$, we specify an ordering on the lateral transvections. Let $s_1 \leq \ldots \leq s_k$ be a total order on the vertices of $S$. For lateral transvections $\tau_{s_ia}, \tau_{s_jb}$, we say $\tau_{s_ia} \leq \tau_{s_jb}$ if $s_i \leq s_j$. For a fixed $i$, we refer to the set $\{ \tau_{s_ia} \mid a \in \Delta \}$ as a \emph{$\Delta$-block}.

We now use properties of the graph $\Delta$ to determine the rest of the ordering on the lateral transvections. Recall that for vertices $x,y \in V$, $x$ dominates $y$ if $\lk(y) \subseteq \st(x)$, and we write $y \leq x$. Charney-Vogtmann \cite{CharneyVogtmann} show that $\leq$ is a pre-order (that is, a reflexive, transitive relation) on $V$, and use it to define the following equivalence relation. Let $v,w \in V$. We say $v$ and $w$ are \emph{domination equivalent} if $v \leq w$ and $w \leq v$. If this is the case, we write $v \sim w$, and let $[v]$ denote the domination equivalence class of $v$.

The pre-order on $V$ descends to a partial order on $V / \sim$. We also denote this partial order by $\leq$. The group $\mathrm{Aut}(\Delta)$ acts on the set of domination classes of $\Delta$. Let $\mathcal{O}$ be the set of orbits of this action, writing $\mathcal{O}_{[v]}$ for the orbit of the class $[v]$. We wish to define a partial order $\ll$ on $\mathcal{O}$ which respects the partial order on the domination classes. That is, if $[v] \leq [w]$, then $\mathcal{O}_{[v]} \ll \mathcal{O}_{[w]}$, for domination classes $[v]$ and $[w]$.

We achieve this by defining a relation $\ll$ on $\mathcal{O}$ by the rule $\mathcal{O}_{[v]} \ll \mathcal{O}_{[w]}$ if and only if there exists $[w'] \in \mathcal{O}_{[w]}$ such that $[v] \leq [w']$. This is well-defined, since $\aut(\Delta)$ acts transitively on each $\mathcal{O}_{[v]} \in \mathcal{O}$. The properties of $\leq$ discussed above give us the following proposition.
\begin{prop}The relation $\ll$ on $\mathcal{O}$ is a partial order. \end{prop}
\begin{proof}We utilise the transitive action of $\aut(\Delta)$ on each $\mathcal{O}_{[v]} \in \mathcal{O}$. The only work lies in establishing the anti-symmetry of $\ll$. This can be achieved by noting that if $[v] \leq [w]$, then $|\st(v)| \leq |\st(w)|$, and if $[v] \leq [w]$ with $|\st(v)| = |\st(w)|$ then $[v]=[w]$.\end{proof}
We use $\ll$ to define a total order on the vertices of $\Delta$, by first extending $\ll$ to a total order on $\mathcal{O}$. We also place total orders on the domination classes within each $\mathcal{O}_{[v]} \in \mathcal{O}$, and on the vertices within each domination class. Now each vertex is relabelled $T(p,q,r)$ to indicate its place in the order: $T(p,q,r)$ is the $r$th vertex of the $q$th domination class of the $p$th orbit. When working with a given $\Delta$-block, we can identify the lateral transvections with the vertices of $\Delta$, allowing us to think of $T(p,q,r)$ as a lateral transvection. Thus, we may think of a specific $\Delta$-block as inheriting an order from the ordering on $\Delta$.

\textbf{The centraliser of the image of $\mathbf{\alpha}$.} We now explicitly determine the image of $\alpha$, and its centraliser, in $\glkdelt$. Looking at how $\glk \times \autraagdelt$ acts on $\zkdelt$ (see the Appendix), we see that the image of $\alpha$ is \[Q := \glk \times \Phi_\Delta,\] where $\Phi_\Delta \leq \gldelt$ is the image of $\autraagdelt$ under the homomorphism induced by abelianising $A_\Delta$. The action of $Q$ on $\zkdelt$ factors through $\mathrm{GL}(k+|\Delta|, \Z)$ via the canonical map $\autraag \to \mathrm{GL}(k+|\Delta|)$. Working in $\mathrm{GL}(k+|\Delta|, \Z)$ instead of $\glkdelt$ is simpler (as pointed out to us by an anonymous referee), however it does not allow us to fully determine the group $C_1$, as the following does.

The matrices in $Q$ have a natural block decomposition given by the $\Delta$-blocks: each $M \in Q$ may be partitioned into $k$ horizontal blocks and $k$ vertical blocks, each of which has size $|\Delta| \times |\Delta|$. We write $M = (A_{ij})$, where $A_{ij}$ is the block entry in the $i$th row and $j$th column. Under this decomposition, we see that the $\glk$ factor of $Q$ is embedded as $$\glk \cong \{ (a_{ij} \cdot I_{|\Delta|}) \mid (a_{ij}) \in \glk \},$$ where $I_{|\Delta|}$ is the identity matrix in $\gldelt$. We write $\mathrm{Diag}(D_1, \ldots, D_k)$ to denote the block diagonal matrix $(B_{ij})$ where $B_{ii} = D_i$ and $B_{ij} = 0$ if $i \neq j$. The $\Phi_\Delta$ factor of $Q$ embeds as $$ \Phi_\Delta \cong \{ \mathrm{Diag}(M, \ldots , M) \mid M \in \Phi_\Delta \} \leq Q.$$

We now determine the centraliser, $C(Q)$, of $Q$ in $\glkdelt$. The proof is similar to the standard computation of $Z(\glk)$.
\begin{lemma}The centraliser $C(Q)$ is a subgroup of $\{ \mathrm{Diag}(M, \ldots , M) \mid M \in \gldelt \}.$
\end{lemma}
\begin{proof}Clearly an element of $C(Q)$ must centralise the $\glk$ factor of $Q$. Let $D$ be the subgroup of diagonal matrices in $\glk$, and define \[ \hat D := \{ (\epsilon_{ij} \cdot I_{|\Delta|}) \mid (\epsilon_{ij}) \in D \} \leq Q. \] Suppose $(A_{ij}) \in C(Q)$ centralises $\hat D$. Then for each $(\epsilon_{ij} \cdot I_{|\Delta|}) \in \hat D$, we must have \[ (A_{ij}) = (\epsilon_{ij} \cdot I_{|\Delta|}) (A_{ij}) (\epsilon_{ij} \cdot I_{|\Delta|}) = (\epsilon_{ii}\epsilon_{jj}A_{ij}),\] since $(\epsilon_{ij} \cdot I_{|\Delta|})$ is block diagonal. Since $\epsilon_{ii} \in \{-1,1\}$ for $1 \leq i \leq k$, we must have $A_{ij} = 0$ if $i \neq j$, so $(A_{ij})$ is block diagonal. By considering which block diagonal matrices centralise $(E_{ij} \cdot I_{|\Delta|})$, where $(E_{ij}) \in \glk$ is an elementary matrix, we see that any block diagonal matrix centralising the $\glk$ factor of $Q$ must have the \emph{same} matrix $M \in \gldelt$ in each diagonal block. It is then a standard calculation to verify that any choice of $M \in \gldelt$ will centralise the $\glk$ factor of $Q$.
\end{proof}

The problem of determining $C(Q)$ has therefore been reduced to determining the centraliser of $\Phi_\Delta$ in $\gldelt$. The total order we specified on the vertices of $\Delta$ gives a block lower triangular decomposition of $M \in \Phi_\Delta$, which we utilise in the proof of Proposition \ref{centraliser}. This builds upon a matrix decomposition given by Day \cite{Day} and Wade \cite{Wade}.

Observe that $\Phi_\Delta$ contains the diagonal matrices of $\gldelt$. As in the above proof, anything centralising $\Phi_\Delta$ must be a diagonal matrix. For a diagonal matrix $E \in \gldelt$, we write $E(p,q,r)$ for the diagonal entry corresponding to the vertex $T(p,q,r)$ of $\Delta$.
\begin{prop}\label{centraliser}A diagonal matrix $E \in \gldelt$ centralises $\Phi_\Delta$ if and only if the following conditions hold: \begin{description}\item[(1)]If $p = p'$, then $E(p,q,r) = E(p',q',r')$, and,
\item[(2)] If $T(p,q,r)$ is dominated by $T(p',q',r')$, then $E(p,q,r) = E(p',q',r')$ \end{description}
\end{prop}
\begin{proof}We define a block decomposition of the matrices in $\gldelt$ using the sizes of the orbits, $\mathcal{O}_{[v_1]} \ll \ldots \ll \mathcal{O}_{[v_l]}$. Let $m_i = |\mathcal{O}_{[v_i]}|$. We partition $M \in \gldelt$ into $l$ horizontal blocks and $l$ vertical blocks, writing $M = (M_{ij})$, where $M_{ij}$ is an $m_i \times m_j$ matrix. Observe that due to the ordering on the lateral transvections, if $i < j$, then $M_{ij} = 0$.

Let $E \in \gldelt$ satisfy the conditions in the statement of the proposition. We may write $E = \mathrm{Diag}(\epsilon_1 \cdot I_{m_1 \times m_1}, \ldots , \epsilon_l \cdot I_{m_l \times m_l}),$ where each $\epsilon_i \in \{-1,1\}$ ($1 \leq i \leq l$). Then $EM = (\epsilon_i \cdot M_{ij})$ and $ME = (\epsilon_j \cdot M_{ij}).$ We see that $ME$ and $EM$ agree on the diagonal blocks, and on the blocks where $M_{ij} = 0$. If $i > j$ and $M_{ij} \neq 0$, then there must be a vertex $T(j,q,r)$ being dominated by a vertex $T(i,q',r')$. By assumption, $\epsilon_i = \epsilon_j$. Therefore $EM = ME$ and $E \in C(Q)$.

Suppose now that $E \in \gldelt$ fails the first condition. Without loss of generality, suppose $E(p,q,1) \neq E(p,q',1)$. Since, by definition, $\Aut(\Delta)$ acts transitively on the elements of $\mathcal{O}_{[v_p]}$, there is some $P \in \gldelt$ induced by some $\phi \in \aut(\Delta)$ which acts by exchanging the $q$th and $q'$th domination classes. A standard calculation shows that $[E,P] \neq 1$.

Finally, suppose $E \in \gldelt$ fails the second condition. Assume that $T(p,q,r)$ is dominated by $T(p',q',r')$, but that $E(p,q,r) \neq E(p',q',r')$. In this case, $E$ fails to centralise the elementary matrix which is the result of transvecting $T(p,q,r)$ by $T(p',q',r')$.
\end{proof}

\textbf{Extending elements of $C(Q)$ to automorphisms of $\autraag$.} Using the map $R$ from section 3.1, for $A \in C(Q) = C_1$ we obtain $R(A) \in \aut(\autraag)$ which acts as $A$ on $\zkdelt$ and as the identity on $\glk \times \autraagdelt$. Note that $R(A)$ acts on $\autraag$ by inverting some collection of lateral transvections: the group $R(C_1)$ is hence a direct sum of finitely many copies of $\Z / 2$.If there are $d$ domination classes in $\Delta$, then $|C_1| \leq 2^d$. We now determine $\hat{R}(C_1)$, the image of $R(C_1)$ in $\out(\autraag)$.

Let $nh \in \zkdelt \rtimes \left [ \glk \times \autraagdelt \right ]$, with $h \neq 1$. Conjugating $\autraag$ by $nh$ fixes $\glk \times \autraagdelt$ pointwise only if $h$ is central in $\glk \times \autraagdelt$. The only such non-trivial central element is $\iota$, the automorphism inverting each generator of $\mathbb{Z}^k$ (see Proposition \ref{centres}). Given that $\alpha_\iota(n) = -n$ for each $n \in \zkdelt$, we see that for any $m \in \zkdelt$, we have $(m, 1)^{(n, \iota)} =(-m,1)$.

So, regardless of which $n$ we choose, conjugation by $n\iota$ is equal to $R(-I_{k|\Delta|})$. In other words, when we conjugate by $n \iota$, we map each lateral transvection to its inverse. Thus, for $A,B \in C_1$, $R(AB^{-1})$ is inner if and only if $A(p,q,r) = - B(p,q,r)$ for every $p,q,$ and $r$. This means $|R(C_1)| = 2 |\hat{R}(C_1)|$.

\textbf{First proof of Theorem A.} We are now able to prove Theorem A for right-angled Artin groups with non-trivial centre.
\begin{proof}[Proof (1) of Theorem A] By Proposition \ref{splitdecomp}, we have a semi-direct product decomposition of $\autraag$, whose kernel is $\zkdelt$. The structure of $C_1 = C(Q)$ is given by Proposition \ref{centraliser}. We have fewest constraints on $C_1$ if $\Delta$ is such that domination occurs only between vertices in the same domination class, and when each domination class lies in an $\Aut(\Delta)$-orbit by itself. This is achieved, for example, if $\Delta = X$, a disjoint union of pairwise non-isomorphic complete graphs, each of rank at least two. Suppose $X$ has $d$ connected components. For $A \in C(Q)$, Proposition \ref{centraliser} implies $A$ is entirely determined by the entries $A(p,1,1)$ ($1 \leq p \leq d$). This gives $|C(Q)| = 2^d$, and so the image of $C(Q)$ in $\out(\autraag)$ has order $2^{d-1}$. As we may choose $d$ to be as large as we like, the result follows.
\end{proof}
\section{Proof of Theorem A: centreless right-angled Artin groups}
In this section, we demonstrate that Theorem A also holds for classes of centreless right-angled Artin groups. From now on, we assume that the graph $\Gamma$ has no social vertices, so that $\raag$, has trivial centre. A simplicial graph $\Gamma = (V,E)$ is said to have \emph{no separating intersection of links (`no SILs')} if for all $v, w \in V$ with $v$ not adjacent to $w$, each connected component of $\Gamma \setminus \left ( \lk(v) \cap \lk(w) \right)$ contains either $v$ or $w$. We have the following theorem. \begin{theorem}[Charney-Ruane-Stambaugh-Vijayan \cite{Charneyetal}] \label{nosils}Let $\Gamma$ be a finite simplicial graph with no SILs. Then $\pc(\raag)$, the subgroup of $\autraag$ generated by partial conjugations, is a right-angled Artin group, whose defining graph has vertices in bijection with the partial conjugations of $\raag$. \end{theorem}

We restrict ourselves to looking at certain no SILs graphs, to obtain a nice decomposition of $\autraag$. We say a graph $\Gamma$ is \emph{weakly austere} if it has trivial symmetry group and no dominated vertices. Note that this is a loosening of the definition of an austere graph: removing a vertex star need no longer leave the graph connected.
\begin{lemma}\label{sils} Let $\Gamma=(V,E)$ be weakly austere and have no SILs. For $c \in V$, let $K_c = |\pi_0(\Gamma \setminus \st(c))|$. Then \[ |\out(\autraag)| \geq 2^{K_c-1}. \] \end{lemma}
\begin{proof}Since $\Gamma$ is weakly austere, the only LS generators which are defined are the inversions and the partial conjugations. Letting $I_\Gamma$ denote the finite subgroup generated by the inversions $\iota_v$ ($v \in V$), we obtain the decomposition \[ \aut(\raag) \cong \pc(\raag) \rtimes I_\Gamma, \] where the inversions act by inverting partial conjugations in the obvious way. Since $\Gamma$ has no SILs, it follows from Theorem \ref{nosils} that $\pc(\raag) \cong A_\Delta$ for some simplicial graph $\Delta$ whose vertices are in bijection with the partial conjugations of $\raag$.

Fix $c \in V$ and let $\{ \gamma_{c,D_i} \mid 1 \leq i \leq K_c \}$ be the set of partial conjugations by $c$. Let $\eta_{c,j}$ be the LS generator of $\aut(A_\Delta)$ which inverts $\gamma_{c,D_j}$, but fixes the other vertex-generators of $A_\Delta$. This extends to an automorphism of $\aut(\raag)$, by specifying that $I_\Gamma$ is fixed pointwise: all that needs to be checked is that the action of $I_\Gamma$ on $\pc(\raag)$ is preserved, which is a straightforward calculation. We abuse notation, and write $\eta_{c,j} \in \aut(\autraag)$.

If $K_c > 1$, we see $\eta_{c,j}$ is not inner. Assume $\eta_{c,j}$ is equal to conjugation by $p \kappa \in \pc(\raag) \rtimes I_\Gamma$. For $\gamma \in \pc(\raag)$, we have $(\gamma,1)^{(p,\kappa)} = (p \gamma^\kappa p^{-1},1)$. Since $\eta_{c,j}(\gamma_{c,D_j}) = {\gamma_{c,D_j}}^{-1}$, an exponent sum argument tells us that $\kappa$ must act by inverting $\gamma_{c,D_j}$, and so $\kappa$ must invert $c$ in $\raag$. However, $\eta_{c,j}$ fixes $\gamma_{c,D_i}$ for all $i \neq j$, by definition, and a similar exponent sum argument implies that $\kappa$ \emph{cannot} invert $c$ in $\raag$. Thus, by contradiction, $\eta_{c,j}$ cannot be inner.

As above, we may choose a subset of $\{ \gamma_{c,D_i} \}$ to invert, and extend this to an automorphism of $\aut(\raag)$. Take two distinct such automorphisms, $\eta_1$ and $\eta_2$. Their difference $\eta_1 \eta_2^{-1}$ is inner if and only if it inverts \emph{every} element of $\{ \gamma_{c,D_i} \}$. Otherwise, we would get the same contradiction as before. A counting argument gives the desired lower bound of $2^{K_c -1}$.
\end{proof}

Observe that if $\Gamma$ is austere, we cannot find a vertex $c$ with $K_c >1$. This is the reason we loosen the definition and consider weakly austere graphs.

\textbf{Second proof of Theorem A.} By exhibiting an infinite family of graphs over which the size of $| \{ \gamma_{c,D_i} \} |$ is unbounded, applying Lemma \ref{sils} will give a second proof of Theorem A.

\begin{proof}[Proof (2) of Theorem A] Fix $t \in \Z$ with $t \geq 3$. Define $e_0 = 0$ and choose $\{e_1 < \ldots < e_t \} \subset \Z^+$ subject to the conditions:
\begin{description}
\item[(1)] For each $0 < i \leq t$, we have $e_{i} - e_{i-1} > 2$, and
\item[(2)] If $i \neq j$, then $e_{i} - e_{i-1} \neq e_{j} - e_{j-1}$.
\end{description}
We use the set $E:= \{e_i\}$ to construct a graph. Begin with a cycle on $e_t$ vertices, labelled $0, 1, \ldots, e_t-1$ in the natural way. Join one extra vertex, labelled $c$, to those labelled $e_i$, for $0 \leq i < t$. We denote the resulting graph by $\Gamma_E$. Figure \ref{spokes} shows an example of such a $\Gamma_E$. \begin{figure}[h!]
    \centering
    \includegraphics[width=2.6in]{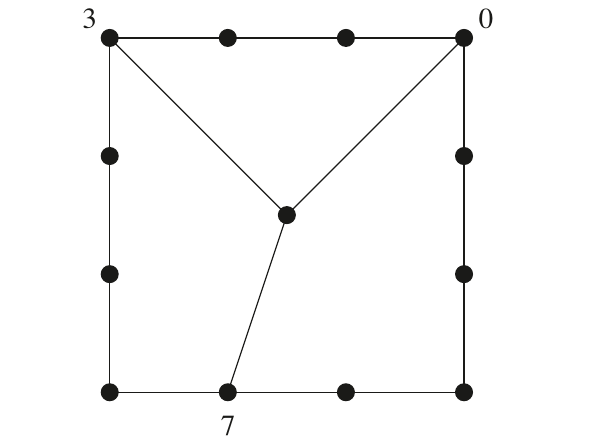}
    \caption{The graph $\Gamma_E$, for $E = \{3,7,12\}$.}
    \label{spokes}
    \end{figure}

For $E \subset \Z^+$ satisfying the above conditions, we see that $\Gamma_E$ is weakly austere and has no SILs. Condition (1) ensures that no vertex is dominated by another. Observe that $c$ is fixed by any $\phi \in \aut(\Gamma_E)$. Since each connected component of $\Gamma \setminus \st(c)$ has $e_{i} - e_{i-1} -1$ elements (for some $1 \leq i \leq t$), condition (2) implies that $\aut(\Gamma_E) = 1$. To see that $\Gamma_E$ has no SILs, observe that the intersection of the links of any two vertices has order at most 1. When a single vertex is removed, $\Gamma_E$ remains connected, and so it has no SILs.

Lemma \ref{sils} applied to the family of graphs $\{ \Gamma_E \}$ proves the theorem. \end{proof}

\section{Extremal behaviour and generalisations}In Sections 3 and 4, we gave examples of $\raag$ for which $\out (\autraag)$ was non-trivial, but not necessarily infinite. Currently, there are very few known $\raag$ for which $\out(\autraag)$ exhibits `extremal behaviour', that is, $\raag$ for which $\out (\autraag))$ is trivial or infinite. In this final section, we discuss the possibility of such behaviour, and generalisations of the current work to automorphism towers.

 \textbf{Complete automorphisms groups.} Recall that a group $G$ is said to be complete if it has trivial centre and every automorphism of $G$ is inner.  Our proofs of Theorems A and B relied upon us being able to exhibit large families of right-angled Artin groups whose automorphisms groups are not complete. It is worth noting that if $\raag$ is not free abelian, then $\autraag$ has trivial centre, and so \emph{a priori}, $\autraag$ \emph{could} be complete.

\begin{prop}\label{centres}Let $A_\Gamma$ be a right-angled Artin group. Then $Z(\mathrm{Aut}(A_\Gamma))$ has order at most two. In particular, if $A_\Gamma$ is not free abelian, then $\mathrm{Aut}(A_\Gamma)$ is centreless. \end{prop}

\begin{proof}For brevity of proof, we assume that $\raag \cong \Z^k \times \raagdelt$, taking $k=0$, and $\Z^k = 1$ if $\raag$ is centreless. If $A_\Gamma$ is free abelian of rank $k$, then $Z(\mathrm{Aut}(A_\Gamma)) \cong Z(\mathrm{GL}(k, \mathbb{Z})) \cong \Z / 2$. From now on, we assume the centre of $\raag$ is proper.

We now adapt the standard proof that a centreless group has centreless automorphism group. Suppose that $\phi \in \mathrm{Aut}(A_\Gamma)$ is central. We know that $\inn(\raag) \cong \raag / \Z^k \cong \raagdelt$. For any $\gamma_w \in \inn(\raag)$, we must have $\gamma_w = \phi \gamma_w \phi^{-1} = \gamma_{\phi(w)}$. So, for $\phi$ to be central, it must fix every element of $\raagdelt$. Observe that if $k = 0$, then $\phi$ must be trivial, and we are done.

Assume now that $k \geq 1$. For any $\phi \in \mathrm{Aut}(A_\Gamma)$, we also have $\phi(u) \in \Z^k$, for all $u \in \Z^k$. So, a central $\phi$ must simply be an element of $\mathrm{GL}(k, \mathbb{Z})$, since it must be the identity on $A_\Delta$, and take $\mathbb{Z}^k$ into itself.

In particular, we have that $Z(\mathrm{Aut}(A_\Gamma)) \leq Z ( \mathrm{GL}(k, \mathbb{Z})) = \{ 1, \iota \}$, where $\iota$ is the automorphism inverting each generator of $\Z^k$. However, lateral transvections are not centralised by $\iota$, and so the centre of $\autraag$ is trivial.
\end{proof}
In this paper, we have focused on finding right-angled Artin groups whose automorphism groups are not complete: an equally interesting question is which right-angled Artin groups \emph{do} have complete automorphism groups, beyond the obvious examples of ones built out of direct products of free groups. In an earlier version of this paper, we conjectured that when $\Gamma$ is austere, the group $\autraag$ would be complete. However, Corey Bregman has since constructed non-trivial members of $\out(\autraag)$ of order 2 when $\Gamma$ is austere; this work will appear in a forthcoming paper.

It may be possible to adapt Bridson-Vogtmann's geometric proof \cite{BridsonVogtmann} of the completeness of $\out(F_n)$ to find examples of $\raag$ for which $\out (\raag)$ is complete, using Charney-Stambaugh-Vogtmann's newly developed outer space for right-angled Artin groups \cite{RAAGspace}.

\textbf{Infinite order automorphisms.} Theorem C showed that it is possible for $\out(\outraag)$ to be infinite, however the question of whether $\out(\autraag)$ can be infinite is still open. An obvious approach to this problem is to exhibit an element $\alpha \in \out(\autraag)$ of infinite order. The approach taken in Section 4, involving graphs $\Gamma$ with no SILs, might seem hopeful, as we certainly know of infinite order non-inner elements of $\aut(\pc (\raag))$: in particular, dominated transvections and partial conjugations. A key property that allowed us to extend $\eta_{c,j} \in \aut(\pc(\raag))$ to an element of $\aut(\autraag)$ was that it respected the natural partition of the partial conjugations by their conjugating vertex. More precisely, $\eta_{c,j}$ sent a partial conjugation by $v \in V$ to a string of partial conjugations, also by $v$. This ensured that the action of $I_\Gamma$ on $\pc(\raag)$ was preserved when we extended $\eta_{c,j}$ to be the identity on $I_\Gamma$.

It might be hoped that we could find a transvection $\tau \in \aut(\pc(\raag))$ which also respected this partition, as $\tau$ could then easily be extended to an infinite order element of $\aut(\autraag)$. However, it is not difficult to verify that whenever $\Gamma$ has no dominated vertices, as in Section 4, no such $\tau$ will be well-defined. Similarly, the only obvious way to extend a partial conjugation $\gamma \in \pc(\pc(\raag))$ is to an element of $\inn(\aut(\raag))$. This leads us to formulate the following open question.

\textbf{Question:} Does there exist a simplicial graph $\Gamma$ such that $\out(\autraag)$ is infinite?

It seems possible that such a $\Gamma$ could exist, however the methods used in this paper do not find one. Our main approach was to find elements of $\aut(\autraag)$ which preserve some nice decomposition of $\autraag$. To find infinite order elements of $\aut(\autraag)$, it may be fruitful, but more unwieldy, to loosen this constraint.

While it is possible to find groups $\autraag$ and $\outraag$ that contain finite index subgroups whose automorphism groups are infinite, it is a difficult problem in general to extend such automorphisms to induce members of $\out(\autraag)$ and $\out(\outraag)$. For example, for $\raag = F_2 \times F_2$, the four well-defined dominated transvections in $\autraag$ generate a finite index copy of $\raag$ inside $\outraag$. However, this copy of $F_2 \times F_2$ lies in $\out(F_2) \times \out(F_2) \leq \outraag$, and the interplay between the dominated transvections and the remaining LS generators prevents any infinite order automorphisms of the copy of $\raag = F_2 \times F_2$ extending in the obvious way. Indeed, $\out(F_2) = \mathrm{GL}(2, \Z)$, and Hua-Reiner \cite{HuaReiner} have already established that the outer automorphism group of this group is finite.

\textbf{Automorphism towers.} Let $G$ be a centreless group. Then $G$ embeds into its automorphism group, $\mathrm{Aut}(G)$, as the subgroup of inner automorphisms, $\mathrm{Inn}(G)$, and $\mathrm{Aut}(G)$ is also centreless. We inductively define $$\mathrm{Aut}^i(G) = \mathrm{Aut(Aut}^{i-1}(G))$$ for $i \geq 0$, with $\mathrm{Aut}^0(G) = G$. This yields the following chain of normal subgroups: \[ G \lhd \mathrm{Aut}(G) \lhd \mathrm{Aut(Aut}(G)) \lhd \ldots \lhd \mathrm{Aut}^i(G) \lhd \ldots, \] which we refer to as the \emph{automorphism tower of} $G$. This sequence of groups is extended transfinitely using direct limits in the obvious way. An automorphism tower is said to \emph{terminate} if there exists a group $A$ in the tower for which the embedding into the next group in the tower is an isomorphism. Observe that a complete group's automorphism tower terminates at the first step. Thomas \cite{Thomas} showed that any centreless group has a terminating automorphism tower, although it may not terminate after a finite number of steps. Hamkins \cite{Hamkins} showed that the automorphism tower of \emph{any} group terminates, although in the above definition, we have only considered automorphism towers of centreless groups.

\textbf{Problem:} Determine the automorphism tower of $\raag$ for an arbitrary $\Gamma$.

This seems a difficult problem in general. A first approach might be to find $\raag$ for which $\out(\autraag)$ is finite. It would then perhaps be easier to study the structure of $\aut^2(\raag)$.

\appendix
\section{Appendix: Conjugating the lateral transvections} Table 1 shows the conjugates of the lateral transvection $\tau_{sa}$ by members of a set $T$ that suffices to generate $\autraag$. We decompose any $\phi \in \aut(\Gamma)$ into its actions on $S$ and $\Delta$. \begin{table}[h!!!!!!]
\centering
\begin{tabular}{|c|c||c|c|} \hline
$\lambda \in T \cup T^{-1}$ & $\lambda \cdot \tau_{sa} \cdot \lambda^{-1}$ & $\lambda \in T \cup T^{-1}$ & $\lambda \cdot \tau_{sa} \cdot \lambda^{-1}$ \\
\hline $\iota_t$ & $\tau_{sa}$ & $\iota_b$ & $\tau_{sa}$ \\
$\iota_s$ &  $-\tau_{sa}$ & $\iota_a$ & $-\tau_{sa}$\\
$\tau_{st}$ & $\tau_{sa}$ & $\tau_{bd}$ & $\tau_{sa}$\\
$\tau_{rt}$ & $\tau_{sa}$ & $\tau_{ab}$ & $\tau_{sa} - \tau_{sb}$\\
$\tau_{ts}$ & $\tau_{sa} + \tau_{ta}$ & $\tau_{ab}^{-1}$ & $\tau_{sa} + \tau_{sb}$\\
$\tau_{ts}^{-1}$ & $\tau_{sa} - \tau_{ta}$ & $\phi \in \aut(\Delta)$ & $\tau_{s\phi(a)}$\\
& & $\gamma_{c,D}$ & $\tau_{sa}$\\ \hline
\end{tabular}
\caption{The conjugates of a lateral transvection $\tau_{sa}$. The vertices $a,b,d,r,s$ and $t$ are taken to be distinct, with $c \in \Delta$ and $D$ being any connected component of $\Gamma \setminus \st(c)$.}
\label{ltconj}
\end{table}


\bibliographystyle{mrl}
\bibliography{compbib}

\end{document}